\documentclass[11pt]{amsart}
\usepackage[all]{xypic}
\usepackage{latexsym}
\usepackage{amssymb}
\usepackage{amsfonts}
\usepackage{amscd}
\usepackage{amsmath,amsthm}
\usepackage{mathrsfs}

\newtheorem{lemma}{Lemma}[section]
\newtheorem{proposition}[lemma]{Proposition}

\newtheorem{Lem}[lemma]{Lemma}
\newtheorem{Prop}[lemma]{Proposition}
\newtheorem{Thm}[lemma]{Theorem}

{

}
\theoremstyle{definition}

\newtheorem{example}[lemma]{Example}
\newtheorem{Def}[lemma]{Definition}

\theoremstyle{remark}

\numberwithin{equation}{section}

\newtheorem*{Rmk}{Remark}

\def\bmat{\begin{pmatrix}}
\def\emat{\end{pmatrix}}
\def\-{\smallsetminus}

\def\~{\widetilde}
\def\ol{\overline}

\def\ge{\geqslant}
\def\le{\leqslant}
\def\phi{\varphi}

\def\deg{\text{deg }}

\def\<{\langle}
\def\>{\rangle}

\def\Gr{\operatorname {Gr}}

\def\GrRep{\operatorname {GrRep}}
\def\Fdim{\operatorname {Fdim}}

\def\QGr{\operatorname {QGr}}

\def\id{\operatorname {id}}

\def\max{\operatorname {max}}

\def\Mod{\operatorname {Mod}}

\def\Ker{\operatorname {Ker}}
\def\Coker{\operatorname {Coker}}

\def\11{\text{\bf 1}}

\def\NN{{\mathbb N}}

\def\ZZ{{\mathbb Z}}

\def\sM{{\mathscr M}}

\title[Weighted path algebras]{Category equivalences involving graded modules over weighted path algebras and weighted monomial algebras}

\begin{document}

\author{Cody Holdaway and Gautam Sisodia}

\address{Department of Mathematics, Box 354350, Univ.
Washington, Seattle, WA 98195}

\email{codyh3@math.washington.edu, gautas@math.washington.edu}



\begin{abstract}
Let $k$ be a field, $Q$ a finite directed graph, and $kQ$ its path algebra. Make $kQ$ an $\NN$-graded algebra by assigning each arrow a positive degree. Let $I$ be an ideal in $kQ$ generated by a finite number of paths
and write $A=kQ/I$. Let $\QGr A$ denote the quotient of the category of graded right $A$-modules modulo the 
Serre subcategory consisting of those graded modules that are the sum of their finite dimensional 
submodules. This paper shows there is a finite directed graph $Q'$ with all its arrows
placed in degree 1 and an equivalence of categories $\QGr A \equiv \QGr kQ'$.  
A result of Smith now implies that $\QGr A \equiv \Mod S$, the category of right modules over an 
ultramatricial, hence von Neumann regular, algebra $S$.
\end{abstract}

\maketitle

\pagenumbering{arabic}

\setcounter{section}{0}

\section{}

\subsection{}
Let $k$ be a field and $A$ an $\NN$-graded $k$-algebra. Let $\Gr A$ be the category with objects the $\ZZ$-graded right $A$-modules and morphisms the degree preserving graded $A$-module homomorphisms. Let $\Fdim A \subseteq \Gr A$ be the localizing subcategory of modules that are the sum of their finite-dimensional submodules. Let $\QGr A$ denote the quotient of $\Gr A$ by $\Fdim A$ and 
$$
\pi^*:\Gr A\to \QGr A
$$
the canonical quotient functor. As $\Fdim A$ is localizing, $\pi^*$ has a right adjoint which we will denote by $\pi_*$. 

For $M \in \Gr A$, let $M(1) \in \Gr A$ be $M$ as a right module with grading given by $M(1)_i:=M_{i+1}$. We call $M(1)$ the {\it shift} of $M$. Shifting determines an auto-equivalence $(1):\Gr A\to \Gr A$. The shift functor descends to an auto-equivalence on $\QGr A$ which we still denote by $(1)$.

\subsection{Question}
\label{qu1}
Let $F=k\langle x_1,\ldots,x_n\rangle$  be the free algebra endowed with an $\NN$-grading 
induced by fixing $\deg(x_i) \ge 1$ for all $i$. What does $\QGr F$ look like? 

The answer when $\deg(x_i)=1$ for all $i$ is in \cite{Sm0}.

\subsection{}
This paper shows that $\QGr F \equiv \QGr kQ$
where $kQ$ is the path algebra of a finite quiver $Q$, graded by giving each arrow degree 1, and defined as follows: view $F$ as the path algebra of the weighted quiver with one vertex and $n$ loops of degrees 
$\deg(x_i)$. If $\deg(x_i)>1$, replace the loop for $x_i$ by $\deg(x_i)-1$ vertices and a cycle through them 
consisting of $\deg(x_i)$ arrows each of degree 1. The
answer to question \ref{qu1}, combined with a result in \cite{Sm1},
says that $\QGr F \equiv \Mod S$ where $S$ is an ultramatricial, hence von Neumann regular, algebra.

The question this paper answers is a little more general.  

\subsection{}

Consider the categories $\QGr A$ where $A$ is a finitely presented $\NN$-graded $k$-algebra belonging to one of the 
following five classes:
\begin{enumerate}
\item[]
\begin{enumerate}
  \item[{\bf PA1:}]
  Path algebras of finite quivers with grading induced by declaring that all arrows have degree 1; this implies that the degree of a path is equal to its length.
  \item[{\bf WPA:}] 
  Weighted path algebras of finite quivers---this is a path algebra with grading given by assigning each arrow a degree $\ge 1$.
  \item 
  [{\bf MA:}]
  Monomial algebras: these are algebras of the form $kQ/I$ where $kQ$ is a weighted path algebra of a finite quiver and $I$ is an
  ideal generated by a finite set of paths.
  \item[{\bf CMA:}]
  Connected monomial algebras: these are monomial algebras $kQ/I$ in which $Q$ has only one vertex.
  \item
  [{\bf CMA1:}]
  Connected monomial algebras that are generated by elements of degree 1. 
\end{enumerate}
\end{enumerate}

The following diagram depicts the inclusions between the five classes: each class is contained in the class ``above'' it.
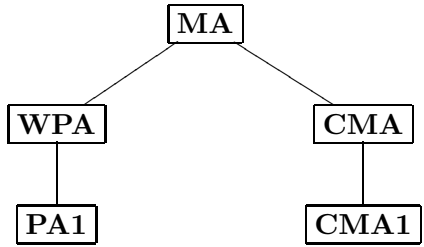
\[
\UseComputerModernTips
\xymatrix{ 
&  &*+[F]{\txt{{\bf MA}}}
 \ar@{-}[dl]  \ar@{-}[dr] 
 \\ 
&*+[F]{\txt{{\bf WPA}} }  \ar@{-}[d] 
&&*+[F]{\txt{{\bf CMA}}}  \ar@{-}[d] 
\\
& *+[F]{\txt{{\bf PA1}}} && *+[F]{\txt{{\bf CMA1}} }
}
\]

\begin{Thm}
\label{thm.main}
If ${\bf C}$ and ${\bf C'}$ are two of the five classes above and $A$ belongs to ${\bf C}$, then there is an algebra $A'$ in ${\bf C'}$ and an equivalence $F:\QGr A \to \QGr A'$ such that $F(M(1))\cong F(M)(1)$ for all $M\in \QGr A$.  
\end{Thm}

We introduce the shorthand $\QGr({\bf C}) \subset \QGr({\bf C'})$ for the result in Theorem \ref{thm.main}.
It is obvious that $\QGr({\bf C}) \subset \QGr({\bf C'})$ if ${\bf C} \subset {\bf C'}$. We depict these obvious inclusions by the diagram
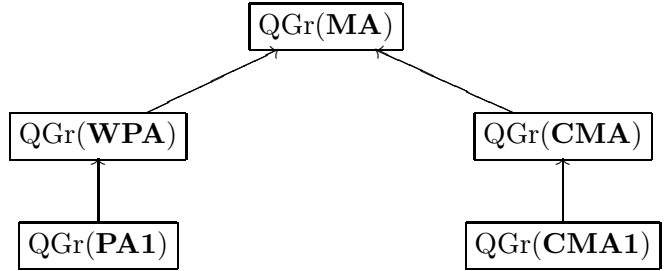
\[
\UseComputerModernTips
\xymatrix{ 
  &*+[F]{\txt{$\QGr({\bf MA})$}} 
 \ar@{-}[dl]  \ar@{-}[dr] 
 \\ 
 *+[F]{\txt{$\QGr({\bf WPA})$} }   \ar[ur]   
&&*+[F]{\txt{$\QGr({\bf CMA})$}}  \ar[ul] 
\\
*+[F]{\txt{$\QGr({\bf PA1})$}}  \ar[u] &&*+[F]{\txt{$\QGr({\bf CMA1})$} } \ar[u] 
}
\]

\subsubsection{Why these five classes}
Free algebras are universal objects in the category of $k$-algebras. Polynomial rings are
universal objects in the category of commutative $k$-algebras.

Why single out monomial algebras? Monomial algebras can often be understood through the combinatorics of words. When dealing with families of algebras presented by generators and relations, monomial relations tend
to turn up as singular points. Monomial relations are surely the simplest relations.

Monomial algebras tend to have infinite global dimension so are less amenable to homological arguments. 
Path algebras always have global dimension $\le 1$ so it is reassuring to know that although $\Gr A$ is homologically awkward, $\QGr A$ is not. 

There seems to be a general consensus that many technicalities can be avoided by assuming the algebra is generated by elements of degree 1.

\subsection{}
\label{ssect.AZ+HS}
Some of the inclusions $\QGr({\bf C})\subset \QGr({\bf C'})$ follow from results already in the literature.

Suppose $A \in {\bf MA}$. Then $A'= k+A_{\ge 1} \in {\bf CMA}$ and  $\dim_k(A/A')<\infty$. So, by \cite[Prop 2.5]{AZ}, $-\otimes_{A'} A$ induces an equivalence $F:\QGr A' \equiv \QGr A$ such that $F(M(1))\cong F(M)(1)$ for all $M\in \QGr A$. Thus 
$$
\QGr({\bf MA})  = \QGr({\bf CMA}),
$$
where the ``$=$'' means $\QGr({\bf MA})\subset \QGr({\bf CMA})$ and $\QGr({\bf CMA})\subset \QGr({\bf MA})$.

Similarly, if $A \in {\bf PA1}$, then $A':= k+A_{\ge 1} \in {\bf CMA1}$ and, by  \cite[Prop 2.5]{AZ}, $\QGr({\bf PA1})  \subset \QGr({\bf CMA1})$. The reverse inclusion  $\QGr({\bf CMA1}) \subset \QGr({\bf PA1})$ holds by \cite[Thm. 1.1]{HS0}.

We depict these results by the diagram
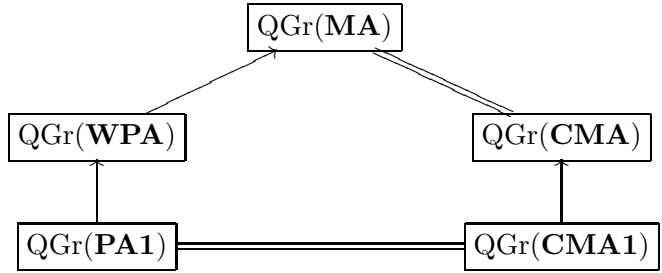
\[
\UseComputerModernTips
\xymatrix{ 
  &*+[F]{\txt{$\QGr({\bf MA})$}} \ar@{=}[dr] 
 \\ 
 *+[F]{\txt{$\QGr({\bf WPA})$} }   \ar[ur]  & &*+[F]{\txt{$\QGr({\bf CMA})$}} 
& 
\\
*+[F]{\txt{$\QGr({\bf PA1})$}}  \ar[u] &&*+[F]{\txt{$\QGr({\bf CMA1})$} } \ar[u]  \ar@{=}[ll]
}
\]

\subsection{}
The new contributions in this paper are Theorem \ref{thm.main2'}, which shows that $\QGr({\bf CMA})  \subset  \QGr({\bf WPA})$,
and Theorem \ref{thm.wpapa1},  which shows that $\QGr({\bf WPA}) \, \subset  \, \QGr({\bf PA1}).$
These two results tell us that  
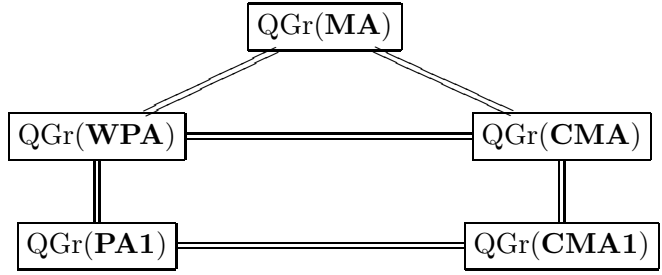
\[
\UseComputerModernTips
\xymatrix{ 
  &*+[F]{\txt{$\QGr({\bf MA})$}} \ar@{=}[dr] 
 \\ 
 *+[F]{\txt{$\QGr({\bf WPA})$} }   \ar@{=}[ur]  \ar@{=}[rr] & &*+[F]{\txt{$\QGr({\bf CMA})$}} 
& 
\\
*+[F]{\txt{$\QGr({\bf PA1})$}}  \ar@{=}[u] &&*+[F]{\txt{$\QGr({\bf CMA1})$} } \ar@{=}[u]  \ar@{=}[ll]
}
\]
and the proof of Theorem \ref{thm.main} is then complete.

\subsection{}
 The proof that $\QGr({\bf CMA})  \subset  \QGr({\bf WPA})$ is similar to the proof 
of \cite[Thm. 1.1]{HS0}. One associates to $A$ in {\bf CMA} a weighted quiver $Q=Q(A)$, the weighted Ufnarovskii 
graph of $A$,
and shows there is a homomorphism of graded algebras $A \to kQ$ whose kernel and cokernel belong to $\Fdim(A)$.  

\subsection{Acknowledgments.}

The authors would like to express their gratitude to S. Paul Smith for reading an early version of this manuscript and providing helpful comments. 

\section{Background, notation, and terminology}
\label{sect.2}      

\subsection{}

Let $Q$ be a finite quiver i.e. a finite directed graph. We denote by $Q_0$, $Q_1$, $s$ and $t$ the vertex set, arrow set, source function and target function of $Q$, respectively.

\begin{Def}
A {\it weighted path algebra} $kQ$ of $Q$ is the path algebra of $Q$ graded by assigning each arrow a degree $\geq 1$. We denote the degree of an arrow $a \in Q_1$ by $\deg(a)$. Define the {\it weight discrepancy} of $kQ$ to be
$$
D(kQ) := \sum_{a \in Q_1} \deg(a) - |Q_1|.
$$ 
\end{Def}

\begin{Rmk}
For $kQ$ a weighted path algebra, $D(kQ) \geq 0$, with equality if and only if all arrows have degree 1.
\end{Rmk}

\begin{Def}
Let $V$ be a $\ZZ$-graded $k$-vector space and $i \in \ZZ$. Define the $\ZZ$-graded $k$-vector space $V(i)$ to be $V$ with grading $V(i)_j = V_{i+j}$ for all $j \in \ZZ$.
\end{Def}

\begin{Def}
Let $V$ and $W$ be $\ZZ$-graded $k$-vector spaces and $i \in \ZZ$. A {\it linear map $f : V \to W$ of degree $i$} is a $k$-linear map $f : V \to W$ such that $f(V_j) \subseteq W_{j + i}$ for all $j \in \ZZ$ (in the case $i = 0$, we say $f$ is {\it degree-preserving}).
\end{Def}

\subsection{}
Let $kQ$ be a weighted path algebra. Recall $\GrRep kQ$, the category of graded representations of $Q$. A graded representation $M= (M_v, M_a)$ of $Q$ is 
\begin{itemize}
\item a graded vector space $M_v$ for each vertex $v$ and
\item a linear map $M_a:M_{s(a)}\to M_{t(a)}$ of degree equal to $\deg(a)$ for each arrow $a$.
\end{itemize}
A morphism $\phi:M\to N$ between two graded representations of $Q$ consists of a degree-preserving linear map $\phi_v:M_v\to N_v$ for each vertex $v$ such that 
$$
\UseComputerModernTips
\xymatrix{ {M_{s(a)}}\ar[r]^{M_a} \ar[d]_{\phi_{s(a)}} & {M_{t(a)}} \ar[d]^{\phi_{t(a)}} \\
           {N_{s(a)}}\ar[r]_{N_a} & {N_{t(a)}} }
$$
commutes for each arrow $a$. There is an equivalence of categories $\Gr kQ\equiv \GrRep kQ$ given by sending a graded module $M$ to the graded representation $(Me_v,M_a)$, where $M_a:Me_{s(a)}\to Me_{t(a)}$ is the linear map determined by the action of the arrow $a$. From now on we identify these two categories.

\section{Proof that $\QGr({\bf WPA}) \, \subset  \, \QGr({\bf PA1})$}      
\label{sect.3}

\subsection{}\label{subsect.3.1}

Let $kQ$ be a weighted path algebra. Suppose $b$ is an arrow in $Q$ with $\deg(b) > 1$. Define a new quiver $Q'$ by declaring 
\begin{align*}
Q'_0 &:= Q_0\sqcup \{z\} \\
Q'_1 &:= (Q_1 - \{b\}) \sqcup \{b' : s(b) \to z, b'' : z \to t(b)\}.
\end{align*}
We define a grading on $kQ'$ by declaring the degree of an arrow $a\in Q'_1 - \{b',b''\}$ to be the degree of $a$ in $kQ$, while $\deg (b')=1$ and $\deg(b'') = \deg(b)-1$. Since $|Q_1'| = |Q_1| + 1$, $D(kQ') = D(kQ)-1$.

\begin{example}
Let $Q$ be the quiver
$$
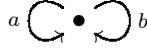
\UseComputerModernTips
\xymatrix{ {\bullet}\ar@(ul,dl)_{a} \ar@(ur,dr)^{b}}
$$
with $\deg(b) > 1$. The quiver $Q'$ obtained by replacing the arrow $b$ as described above is
$$
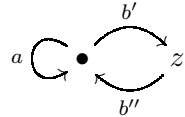
\UseComputerModernTips
\xymatrix{ {\bullet} \ar@(ul,dl)_{a} \ar@/^1pc/[r]^{b'} & z \ar@/^1pc/[l]^{b''} }
$$

with $\deg(b')=1$ and $\deg(b'') =\deg(b)-1$.
\end{example}

\begin{example}
Let $Q$ be the quiver
$$
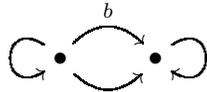
\UseComputerModernTips
\xymatrix{ {\bullet} \ar@(ul,dl) \ar@/^1pc/[r]^{b} \ar@/_1pc/[r] & {\bullet}\ar@(ur,dr)  }
$$
with $\deg(b) > 1$. Replacing $b$ gives the quiver
$$
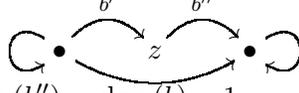
\UseComputerModernTips
\xymatrix{ {\bullet}\ar@(ul,dl) \ar@/^1pc/[r]^{b'} \ar@/_1pc/[rr] & z \ar@/^1pc/[r]^{b''} & {\bullet}\ar@(ur,dr)}
$$
with $\deg(b') = 1$ and $\deg(b'') = \deg(b) - 1$.
\end{example}

\subsection{An adjoint pair}

Let $Q$ and $Q'$ be quivers related as in section \ref{subsect.3.1}. We define a functor $F : \Gr kQ \to \Gr kQ'$ as follows. For $M \in \Gr kQ$, let $F(M) \in \Gr kQ'$ be the representation given by
\begin{itemize}
\item  $F(M)_v:=M_v$ for all $v\in Q'_0 - \{z\}$,
\item $F(M)_z:=M_{s(b)}(-1)$,
\end{itemize}
and for the arrows
\begin{itemize}
\item $F(M)_a:=M_a$ for all $a\in Q'_1 - \{b',b''\}$,
\item $F(M)_{b'} = \id:M_{s(b)}\to M_{s(b)}(-1)$ considered a degree one linear map, and
\item $F(M)_{b''} = M_b : M_{s(b)}(-1)\to M_{t(b)}$ considered a degree $\deg(b)-1$ linear map.
\end{itemize}

Let $\phi:M\to N$ be a morphism in $\Gr kQ$. Define $F(\phi):F(M)\to F(N)$ to be 
\begin{itemize}
\item $F(\phi)_v:=\phi_v$ for all $v\in Q'_0 - \{z\}$ and
\item $F(\phi)_z:=\phi_{s(b)}(-1) : M_{s(b)}(-1) \to N_{s(b)}(-1)$.
\end{itemize}

Since the diagram
$$
\UseComputerModernTips
\xymatrix{ {M_{s(b)}}\ar[d]_{\phi_{s(b)}} \ar[r]^(0.4){\id} & {M_{s(b)}(-1)}\ar[r]^(0.55){M_b} \ar[d]^{\phi_{s(b)}(-1)} & {M_{t(b)}} \ar[d]^{\phi_{t(b)}} \\
           {N_{s(b)}}\ar[r]_(0.45){\id} & {N_{s(b)}(-1)} \ar[r]_(0.55){N_b} & {N_{t(b)}} }
$$
commutes and 
$$
\UseComputerModernTips
\xymatrix{ {M_{s(a)}}\ar[r]^{M_a} \ar[d]_{\phi_{s(a)}} & {M_{t(a)}}\ar[d]^{\phi_{t(a)}} \\
           {N_{s(a)}}\ar[r]_{N_a} & {N_{t(a)}} }           
$$
commutes for all $a \in Q_1' - \{b',b''\}$, $F(\phi)$ is a morphism in $\Gr kQ'$. From the construction, $F(\id_M)=\id_{F(M)}$ and $F(\phi\circ \psi)=F(\phi)\circ F(\psi)$ for any pair of composable morphisms. Hence, we get a functor
$$
F:\Gr kQ \to \Gr kQ'.
$$
Since $F$ preserves kernels and cokernels it is an exact functor. Moreover, $F$ respects shifting, i.e. $F(M(1))\cong F(M)(1)$ for all $M\in \Gr kQ$.

Define a functor $G: \Gr kQ' \to \Gr kQ$ as follows. For $N \in \Gr kQ'$, let $G(N) \in \Gr kQ$ be the representation given by
\begin{itemize}
\item $G(N)_v:=N_v$ for all $v\in Q_0=Q'_0-\{z\}$,
\end{itemize}
and for the arrows
\begin{itemize}
\item $G(N)_a:=N_a$ for all $a\in Q_1-\{b\}$, and
\item $G(N)_b:=N_{b''}\circ N_{b'}$ (a linear map of degree equal to $\deg(b)$).
\end{itemize} 

For $\psi:N\to P$ a morphism in $\Gr kQ'$, define $G(\psi):G(N)\to G(P)$ to be $G(\psi)_v:=\psi_v$ for all $v\in Q_0$, which is a morphism in $\Gr kQ$. Since $G(\id_N)=\id_{G(N)}$ and $G(\phi \circ \psi)=G(\phi)\circ G(\psi)$, we have a functor $G:\Gr kQ' \to \Gr kQ$.

It is evident from the definitions that $G\circ F=\id_{\Gr kQ}$. 

Let $N \in \Gr kQ'$. The module $FG(N)$ is given by the data
\begin{itemize}
\item $FG(N)_v=N_v$ for $v\neq z$, 
\item $FG(N)_z=N_{s(b)}(-1)$,
\item $FG(N)_a=N_a$ for $a\in Q'_1-\{b',b''\}$, 
\item $FG(N)_{b'} = \id : N_{s(b)}\to N_{s(b)}(-1)$ considered a degree one linear map, and
\item $FG(N)_{b''} = N_{b''}\circ N_{b'} : N_{s(b)}(-1)\to N_{t(b)}$ considered a degree $\deg(b) - 1$ linear map.   
\end{itemize}

For each $N \in \Gr kQ'$, define $\epsilon_N:FG(N)\to N$ as follows:
\begin{itemize}
\item $(\epsilon_N)_v = \id : FG(N)_v=N_v\to N_v$ for $v\in Q'_0-\{z\}$, and
\item $(\epsilon_N)_z = N_{b'}$ considered a degree zero linear map  from $FG(N)_z=N_{s(b)}(-1)$ to $N_z$.
\end{itemize}

\begin{Prop}
The correspondence $N \mapsto \epsilon_N$ is a natural transformation $\epsilon : FG \to \id_{\Gr kQ'}$.
\end{Prop}
\begin{proof}
First we show that $\epsilon_N$ is a morphism for each $N \in \Gr kQ'$. For $a\in Q'_1-\{b',b''\}$, $s(a)\neq z$ and $t(a)\neq z$ and the diagram
$$
\UseComputerModernTips
\xymatrix@C=60pt{ {FG(N)_{s(a)}}\ar[d]_{(\epsilon_N)_{s(a)}} \ar[r]^{FG(N)_a} & {FG(N)_{t(a)}} \ar[d]^{(\epsilon_N)_{t(a)}} \\
           {N_{s(a)}}\ar[r]_{N_a} & {N_{t(a)}} }
$$
commutes as $FG(N)_{s(a)}=N_{s(a)}$, $FG(N)_{t(a)}=N_{t(a)}$, the maps $(\epsilon_N)_{s(a)}$ and $(\epsilon_N)_{t(a)}$ are identities, and $FG(N)_a=N_a$. Consider the diagrams
$$
\UseComputerModernTips
\xymatrix@C=60pt{ {FG(N)_{s(b')}}\ar[d]_{(\epsilon_N)_{s(b')}} \ar[r]^{FG(N)_{b'}} & {FG(N)_z}\ar[d]^{(\epsilon_N)_{z}} \\
           {N_{s(b')}}\ar[r]_{N_{b'}} & {N_z}}
$$
and
$$
\UseComputerModernTips
\xymatrix@C=60pt{{FG(N)_z}\ar[d]_{(\epsilon_N)_{z}} \ar[r]^(0.45){FG(N)_{b''}} & {FG(N)_{t(b'')}}\ar[d]^{(\epsilon_N)_{t(b'')}} \\
           {N_z}\ar[r]_{N_{b''}} & {N_{t(b'')}.}} 
$$
The first diagram commutes since $(\epsilon_N)_{s(b')}=\id$, $(\epsilon_N)_z=N_{b'}$ and $FG(N)_{b'}=\id$. The second diagram commutes because $(\epsilon_N)_z=N_{b'}$, $(\epsilon_N)_{t(b'')}=\id$, and $FG(N)_{b''}=N_{b''}\circ N_{b'}$. Hence $\epsilon_N:FG(N)\to N$ is a morphism in $\Gr kQ'$.

Let $\psi:M\to N$ be a morphism in $\Gr kQ'$ and consider the diagram
\begin{eqnarray} \label{diagram.epsnat}
\xymatrix{ {FG(M)}\ar[r]^(0.6){\epsilon_M} \ar[d]_{FG(\psi)} & {M}\ar[d]^{\psi} \\
           {FG(N)}\ar[r]_(0.6){\epsilon_N} & {N}. }
\end{eqnarray}
If $v\in Q'_0-\{z\}$, then
$$
(\epsilon_N)_v\circ FG(\psi)_v=(\epsilon_N)_v\circ \psi_v=\psi_v=\psi_v\circ (\epsilon_M)_v
$$
because $FG(\psi)_v=G(\psi)_v=\psi_v$ and $(\epsilon_N)_v$ and $(\epsilon_M)_v$ are identities. Since $\psi$ is a morphism from $M$ to $N$, $\psi_z\circ M_{b'}=N_{b'}\circ \psi_{s(b')}(-1)$ i.e. $(\epsilon_N)_z\circ FG(\psi)_z=\psi_z\circ (\epsilon_M)_z$. Hence, diagram (\ref{diagram.epsnat}) commutes and $\epsilon:FG\to \id_{\Gr kQ'}$ is a natural transformation.
\end{proof}

\begin{Prop}
The functor $F$ is left adjoint to $G$.
\end{Prop}
\begin{proof}
Let $\eta:\id_{\Gr kQ}\to GF$ be the identity transformation and $\epsilon:FG\to \id_{\Gr kQ'}$ the transformation constructed above. For $M\in \Gr kQ$,
$$
(\epsilon F\cdot F\eta)_M=\epsilon_{F(M)}\circ F(\eta_M)=\epsilon_{F(M)}\circ \id_{F(M)}=\epsilon_{F(M)}.
$$
If $v\in Q'_0-\{z\}$, then $(\epsilon_{F(M)})_v=\id_{M_v}$. The linear map $F(M)_{b'}:M_{s(b)}\to M_{s(b)}(-1)$ is the identity  considered a linear map of degree one. Hence $(\epsilon_{F(M)})_z=F(M)_{b'}:F(M)_{s(b)}(-1)\to F(M)_z=M_{s(b)}(-1)$ considered  a linear map of degree zero i.e. $(\epsilon_{F(M)})_z=\id_{M_{s(b)}(-1)}$. Therefore $(\epsilon F\cdot F\eta)_M=\id_{F(M)}$ which shows that $(\epsilon F\cdot F\eta)$ is the identity natural transformation $F\to F$. 

For $N\in \Gr kQ'$,
$$
(G\epsilon \cdot \eta G)_N=G(\epsilon_N)\circ \eta_{G(N)}=G(\epsilon_N).
$$
For any vertex $v\in Q_0=Q'_0-\{z\}$, $G(\epsilon_N)_v=(\epsilon_N)_v=\id_{N_v}$. Hence $(G\epsilon \cdot \eta G):G\to G$ is the identity natural transformation.

Therefore $(F,G)$ is an adjoint pair with $\eta$ and $\epsilon$ the unit and counit respectively.
\end{proof}

\subsection{The induced equivalence}

We will keep the same notation as before. Let $\pi^*:\Gr kQ'\to \QGr kQ'$ be the quotient functor with right adjoint $\pi_*$.

Let $\sigma:\id_{\Gr kQ'} \to \pi_*\pi^*$ be the unit and $\tau:\pi^*\pi_*\to \id_{\QGr kQ'}$ the counit of the adjoint pair $(\pi^*,\pi_*)$. By \cite[Prop. 4.3, pg. 176]{Pop} the counit $\tau$ is a natural isomorphism. Since $F$ is left adjoint to $G$ and $\pi^*$ is left adjoint to $\pi_*$, $\pi^*F$ is left adjoint to $G\pi_*$. Moreover, the unit and counit of this adjoint pair are given by
\begin{eqnarray} \label{eqn.adpair}
\text{unit }\quad G\sigma F\cdot \eta&:&\id_{\Gr kQ}\to G\pi_*\circ \pi^*F \\
\text{counit }\quad \tau \cdot \pi^*\epsilon \pi_*&:&\pi^*F\circ G\pi_* \to \id_{\QGr kQ'}. \nonumber
\end{eqnarray}
As $F$ and $\pi^*$ are exact, so is $\pi^*F:\Gr kQ \to \QGr kQ'$. 

\begin{Lem} \label{lem.kerisfdim}
Using the same notation as above,
$$
\Ker \pi^*F=\Fdim kQ.
$$
\end{Lem}
\begin{proof}
If $M \in \Gr kQ$ is finite-dimensional then so is $F(M)$. Suppose $M \in \Fdim kQ$, i.e. $M$ is a direct limit of finite-dimensional modules. Since $F$ is a left adjoint it preserves direct limits, so $F(M)$ is a direct limit of finite-dimensional $kQ'$ modules, i.e.  $F(M)\in \Fdim kQ'$. Therefore, $\pi^*F(M)=0$ and we see $\Fdim kQ \subseteq \Ker \pi^*F$.

Suppose $M \in \Ker \pi^*F$, i.e. $F(M) \in \Fdim kQ'$. Define a degree zero $k$-linear map $f : kQ \to kQ'$ by 
\begin{itemize}
\item $f(v) = v$ for all $v \in Q_0$,
\item $f(a) = a$ for all $a \in Q_1 - \{b\}$, and
\item $f(b) = b'b''$ 
\end{itemize}
and extend multiplicatively.

Define a degree zero linear map $g : M \to F(M)$ by
$$
g \mid_{M_v} = \id : M_v \to F(M)_v = M_v
$$
for all $v \in Q_0$.

Pick $m \in M$. Since $F(M) \in \Fdim kQ'$, there exists an $N\in\NN$ such that $g(m).p = 0$ for all paths $p$ in $Q'$ such that $\deg(p) \geq N$. Let $q$ be a path in $Q$ such that $\deg(q) \geq N$. Since $\deg(f(q))= \deg(q)$,
$$
g(m.q) = g(m).f(q) = 0.
$$
Since $g$ is injective, $m.q = 0$, hence $M \in \Fdim kQ$ and $\Ker \pi^*F \subseteq \Fdim kQ$. The result follows.
\end{proof}

\begin{Prop} \label{prop.episomodtors}
The natural transformation $\epsilon:FG\to \id_{\Gr kQ'}$ is an isomorphism modulo torsion, i.e. $\pi^*(\epsilon_M)$ is an isomorphism for all $M \in \Gr kQ'$.
\end{Prop}
\begin{proof}
Let $M\in \Gr kQ'$ and consider the exact sequence
\begin{equation}\label{es}
\xymatrix{ {0}\ar[r] & {\Ker \epsilon_M}\ar[r] & {FG(M)}\ar[r]^-{\epsilon_M} & {M}\ar[r] & {\Coker \epsilon_M}\ar[r] & {0}. }
\end{equation}
For each vertex $v\in Q'_0\-\{z\}$, $(\epsilon_M)_v=\id_{M_v}$ so  $(\Ker \epsilon_M)_v$ and $(\Coker \epsilon_M)_v$ are zero. Hence, the modules $\Ker \epsilon_M$ and $\Coker \epsilon_M$ are supported only at the vertex $z$ so every arrow in $Q'$ acts trivially on $\Ker \epsilon_M$ and $\Coker \epsilon_M$. Thus both $\Ker \epsilon_M$ and $\Coker \epsilon_M$ are torsion. Applying the exact functor $\pi^*$ to the exact sequence (\ref{es}) gives the exact sequence
$$
\xymatrix{ {0}\ar[r] & {0}\ar[r] & {\pi^*FG(M)}\ar[r]^-{\pi^*(\epsilon_M)} & {\pi^*M}\ar[r] & {0}\ar[r] & {0} },
$$ 
thereby showing $\pi^*(\epsilon_M)$ is an isomorphism.
\end{proof}

\begin{Thm} \label{thm.1stepequiv}
The adjoint pair of functors $(F,G):\Gr kQ\to \Gr kQ'$ induces an equivalence
$$
\QGr kQ \equiv \QGr kQ'.
$$
Moreover, this equivalence respects shifting.
\end{Thm}
\begin{proof}
Let $\sM$ be an object in $\QGr kQ'$. By Proposition \ref{prop.episomodtors}, $\pi^*(\epsilon_{\pi_*(\sM)})$ is an isomorphism. Hence,
$$
(\tau \cdot \pi^*\epsilon \pi_*)_{\sM}=\tau_{\sM}\circ \pi^*(\epsilon_{\pi_*(\sM)})
$$
is an isomorphism since $\tau_{\sM}$ is an isomorphism. Therefore, $\pi^*F$ is an exact functor with a right adjoint for which the counit of the adjunction is an isomorphism, i.e. the right adjoint is fully faithful. By \cite[Theorem 4.9, pg. 180]{Pop}, $\pi^*F$ induces an equivalence
$$
\frac{\Gr kQ}{\Ker \pi^*F}\equiv \QGr kQ'.
$$
This induced equivalence respects shifting. Lemma \ref{lem.kerisfdim} shows $\Ker \pi^*F=\Fdim kQ$ which finishes the proof.
\end{proof}

\begin{Thm}[$\QGr({\bf WPA}) \, \subset  \, \QGr({\bf PA1})$]\label{thm.wpapa1}
Let $kQ$ be a weighted path algebra. There is a path algebra $k \ol{Q}$ generated in degree one and an equivalence $\QGr kQ \equiv \QGr k\ol{Q}$ which respects shifting.  
\end{Thm}
\begin{proof}
We induct on $D(kQ)$. If $D(kQ)=0$, that is, all arrows in $kQ$ have degree one, there is nothing to prove. 

Suppose $D(kQ) > 0$. Let $Q'$ be a quiver obtained by replacing an arrow of degree greater than one in the manner of section \ref{sect.3}. Since $D(kQ') = D(kQ) - 1$, by induction there is a quiver $\ol{Q}$ with all arrows in degree one and an equivalence $\QGr kQ'\equiv \QGr k\ol{Q}$ that respects shifting. Hence, there is an equivalence $\QGr kQ\equiv \QGr k\ol{Q}$ which respects shifting by Theorem \ref{thm.1stepequiv}.
\end{proof}

\begin{example}
Let $F=k\<x_1,x_2,x_3\>$ with $\deg x_i=i$. For $Q$ the quiver
$$
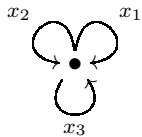
\UseComputerModernTips
\xymatrix{ {\bullet} \ar@(u,r)^{x_1} \ar@(u,l)_{x_2} \ar@(dl,dr)_{x_3} }
$$
with $\deg x_i = i$, $F=kQ$. 
Successively replacing an arrow of degree greater than one gives the quiver $\ol{Q}$
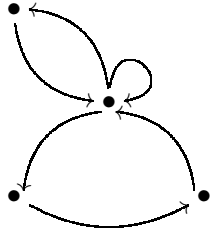
$$
\UseComputerModernTips
\xymatrix{ {\bullet}\ar@/_1pc/[dr] & {} & {} \\
           {} & {\bullet}\ar@/_1pc/[ul] \ar@(u,r) \ar@<1ex>@/_1pc/[dl] & {} \\
           {\bullet}\ar@/_1pc/[rr] & {} & {\bullet}\ar@<1ex>@/_1pc/[ul] \\
           {} & {} & {} }
$$
with all arrows in degree one. By Theorem \ref{thm.wpapa1} there is an equivalence $\QGr F \equiv \QGr k\ol{Q}$. This example illustrates an answer to Question (\ref{qu1}).
\end{example}

\section{Proof that $\QGr(\bf{CMA})\, \subset \, \QGr({\bf WPA})$}

\subsection{The Ufnarovski graph of a monomial algebra} 

In \cite{U1}, V. Ufnarovskii associates to any connected monomial algebra $A$ a graph which has the same growth.
 
Let $A=k\<G\>/(F)$ where $G$ is a finite set of letters and $F$ is a finite set of words in those letters. Every connected monomial algebra can be written in such a way. Following \cite{HS0} words in $F$ are said to be {\it forbidden} while words in $(F)$ are called {\it illegal}. Words not in $(F)$ are called {\it legal}. The set of legal words is denoted $L$.

The length (not to be confused with degree) of a word $w$ is the number of letters in it and is denoted $|w|$. We write $L_n$ for the set of legal words of length $n$. Let $\ell+1$ be the maximum length of a forbidden word:
$$
\ell+1:=\max\{|w|\;|\; w\in F\}.
$$
The Ufnarovskii graph of $A$ is denoted $Q(A)$, or just $Q$ if $A$ is clear from context, and is defined as follows:
\begin{eqnarray*}
Q(A)_0&:=& L_{\ell} \\
Q(A)_1&:=& L_{\ell+1} \\
s(w)&:=& \text{ the unique word in $L_{\ell}$ such that }w\in s(w)G \\
t(w)&:=& \text{ the unique word in $L_m$ such that }w\in Gt(w). 
\end{eqnarray*} 
To elaborate on the last two lines in the definition of $Q(A)$, given a legal word $w$ of length $\ell+1$, there are unique words $v,u\in Q(A)_0=L_\ell$ and unique letters $x,y\in G$ such that
$$
w=vy=xu.
$$
Hence $s(w)=v$ and $t(w)=u$. When a word $w\in L_{\ell+1}$ is treated as an arrow, we will often write $\vec{w}$.

\begin{example} \label{ex.ufgraph}
Let 
$$
A=\frac{k\<x,y,z\>}{(x^2,yx,zy,xz,z^2,y^4)}.
$$
Here, $\ell+1=4$. The legal words of length $3$ and $4$ are
\begin{eqnarray*}
Q(A)_0=L_3&=&\{xy^2,xyz,y^2z,yzx,zxy,y^3\} \\
Q(A)_1=L_4&=&\{xy^2z,xyzx,y^2zx,yzxy,zxy^2,zxyz,y^3z,xy^3\}.
\end{eqnarray*}
Thus, the Ufnarovskii graph is
$$
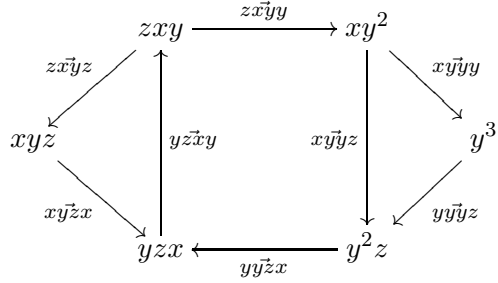
\UseComputerModernTips
\xymatrix{ {} & {zxy}\ar[rr]^{\vec{zxyy}} \ar[dl]_{\vec{zxyz}} & {} & {xy^2}\ar[dd]_{\vec{xyyz}} \ar[dr]^{\vec{xyyy}} & {} \\
           {xyz}\ar[dr]_{\vec{xyzx}} & {} & {} & {} & {y^3}\ar[dl]^{\vec{yyyz}} \\
           {} & {yzx}\ar[uu]_{\vec{yzxy}} & {} & {y^2z}\ar[ll]^{\vec{yyzx}} & {} }
$$
\end{example}

\subsection{Labeling the arrows.}

We label arrows in $Q(A)$ by elements in $G$. The label attached to an arrow $\vec{w}$ is the first letter of $w$. For example, the label attached to $\vec{zxyy}$ is $z$. We extend this labeling to paths. The label for $\vec{w_1}\cdots \vec{w_k}$ is $x_1\cdots x_k$ where $x_i$ is the label of $\vec{w_i}$. For example, the label attached to 
$$
(\vec{zxyy})(\vec{xyyy})(\vec{yyyz})(\vec{yyzx})(\vec{yzxy})(\vec{zxyz})
$$
is $zxyyyz$.

\begin{example}
The labeling for the Ufnarovskii graph in example \ref{ex.ufgraph} is 
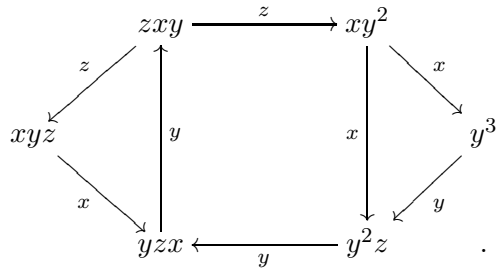
$$
\UseComputerModernTips
\xymatrix{ {} & {zxy}\ar[rr]^{z} \ar[dl]_{z} & {} & {xy^2}\ar[dd]_{x} \ar[dr]^{x} & {} \\
           {xyz}\ar[dr]_{x} & {} & {} & {} & {y^3}\ar[dl]^{y} \\
           {} & {yzx}\ar[uu]_{y} & {} & {y^2z}\ar[ll]^{y} & {.} }
$$
\end{example}

Suppose $A$ and $kQ(A)$ are graded by declaring $\deg(G)=1$ and $\deg(Q(A)_1)=1$. It is shown in \cite{HS0} that there is a graded algebra homomorphism $f:A\to kQ(A)$ defined by
\begin{equation} \label{eqn.f}
f(x)=\sum{\vec{w}}
\end{equation}
for $x\in G$ where the sum is over all arrows $\vec{w}$ labeled $x$. If there are no arrows labeled $x$, then $f(x)=0$. This is used to prove the following theorem in \cite{HS0, HS1}.

\begin{Thm} \label{thm.HS0main}
Let $A$ be a monomial algebra generated in degree one, $Q(A)$ its Ufnarovskii graph and $f:A\to kQ(A)$ the morphism discussed above. Then $-\otimes_{A} kQ(A)$ induces an equivalence of categories
$$
\QGr A \equiv \QGr kQ(A).
$$ 
\end{Thm}
 
\subsection{The weighted Ufnarovskii graph of a weighted monomial algebra.}

Suppose $A=k\<G\>/(F)$ is a connected monomial algebra. For any arrow $\vec{w}$ in $Q(A)$ labeled $x$, define $\deg \vec{w}=\deg x$. We call $Q(A)$ with this new grading the {\it weighted Ufnarovskii graph} associated to $A$. The morphism $f:A\to kQ(A)$ defined in equation (\ref{eqn.f}) sends $x\in G$ to the sum of all arrows labeled $x$ in $kQ(A)$. Hence, $f$ is a morphism of graded algebras.  

Let $A$ be a connected monomial algebra with $Q(A)$ its Ufnarovskii graph. Whatever grading we give to $A$ as long as we give $kQ(A)$ the grading by declaring $\deg(\vec{w})=\deg(x)$ where $x$ is the label of $\vec{w}$, $\Ker f$ and $\Coker f$ will be graded modules. 

If we grade $A$ and $kQ(A)$ by putting all generators in degree one, it is shown in \cite{HS0} that $\Ker f$ and $\Coker f$ are in $\Fdim A$. Therefore, every element in either of these modules generates a finite-dimensional (ungraded) submodule. 

If we now give $A$ a grading with some of the generators in degrees greater than one and regrade $kQ(A)$ accordingly, $\Ker f$ and $\Coker f$ are the same modules as before except with different gradings. As every element in either $\Ker f$ or $\Coker f$ generates a finite-dimensional submodule, $\Ker f$ and $\Coker f$ will be in $\Fdim A$ with this new grading. Using Proposition 2.5 in \cite{AZ} we get the following proposition.

\begin{proposition} \label{prop.main2}
Let $A$ be a connected monomial algebra with $Q(A)$ its weighted Ufnarovskii graph. Then $-\otimes_AkQ(A)$ induces an equivalence of categories
$$
\QGr A \equiv \frac{\Gr kQ(A)}{T_A}
$$
where $T_A$ is the localizing sub-category of $\Gr kQ(A)$ consisting of all modules whose restriction to $A$ is in $\Fdim A$.
\end{proposition}  

Let $p$ be a path in $Q(A)$ with labeling $x_{j_1}\ldots x_{j_r}$, say
$$
\UseComputerModernTips
\xymatrix{ {v_0}\ar[r]^{x_{j_1}} & {\cdots} \ar[r]^{x_{j_r}} & {v_r}, }
$$
and write $v_r=x_{j_{r+1}}\ldots x_{j_{r+\ell}}$. By \cite[Lemma 3.1]{HS0}, $v_{i-1}=x_{j_i}\ldots x_{j_{i+\ell-1}}$. Hence, the path $p$ is completely determined by its labeling and its ending vertex. In other words, different paths with the same labeling end at different vertices.

\begin{lemma} \label{lem.gradedgen}
Let $A$ be a connected monomial algebra with $Q(A)$ its weighted Ufnarovskii graph. For all $n\geq 0$,
$$
kQ(A)_n=f(A_n)kQ_0.
$$
\end{lemma}
\begin{proof}
Let $p$ be a path of degree $n$ with label $x_{j_1}\cdots x_{j_r}$ and let $v$ be its target. By the previous discussion, $p$ is the only path with this label which ends at $v$. As
$$
f(x_{j_1}\cdots x_{j_r})=\sum{q}
$$  
where the sum is over all paths labeled $x_{j_1}\ldots x_{j_r}$,  
$$
f(x_{j_1}\ldots x_{j_r})e_v=p
$$
which shows $p\in f(A_n)kQ(A)_0$. As all paths of degree $n$ form a basis for $kQ(A)_n$ the lemma follows.
\end{proof}

\begin{Thm}[$\QGr({\bf CMA})\, \subset \, \QGr({\bf WPA})$] \label{thm.main2'}
Let $A$ be a connected monomial algebra with $Q(A)$ its weighted Ufnarovskii graph. Then $-\otimes_AkQ(A)$ induces an equivalence
$$
F:\QGr A \equiv \QGr kQ(A).
$$
Moreover, $F(M(1))\cong F(M)(1)$ for all $M\in \QGr A$.
\end{Thm}
\begin{proof}
By Proposition \ref{prop.main2} we just need to show $T_A=\Fdim kQ(A)$. 

Let $M$ be a $kQ(A)$-module also considered an $A$-module via $f$. Let $m\in M$. By Lemma \ref{lem.gradedgen}, $mA_n=0$ if and only if $mkQ(A)_n=0$. Hence, $M\in \Fdim kQ(A)$ if and only if $M\in \Fdim A$. Therefore, $T_A=\Fdim kQ(A)$.
\end{proof}

\begin{example}
Let $A$ be the connected monomial algebra
$$
A=\frac{k\<x,y\>}{(yx,x^3)}
$$
where $\deg(x)=1$ and $\deg(y)=2$. The sets of legal words of length $2$ and $3$ are
\begin{eqnarray*}
Q(A)_0=L_2&=&\{x^2,xy,y^2\} \\
Q(A)_1=L_3&=&\{x^2y, y^3,xy^2\}.
\end{eqnarray*}
Hence, the weighted Ufnarovskii graph $Q(A)$ is given by
$$
\UseComputerModernTips
\xymatrix{ {x^2}\ar[r]^{\vec{xxy}} & {xy}\ar[r]^{\vec{xyy}} & {y^2}\ar@(ur,dr)^{\vec{yyy}} }
$$ 
with $\deg(\vec{xxy})=\deg(\vec{xyy})=1$ and $\deg(\vec{yyy})=2$.

Replacing the degree $2$ arrow with two degree one arrows yields the quiver $Q'$
$$
\UseComputerModernTips
\xymatrix{ {x^2}\ar[r] & {xy}\ar[r] & {y^2}\ar@/^1pc/[r] & {z}\ar@/^1pc/[l] }
$$
with all arrows in degree one. The monomial subalgebra $k+kQ'_{\geq 1}$ can be presented as $B=k\<x_1,x_2,x_3,x_4\>/I$ where 
$$
I=(x_1^2, x_1x_3, x_1x_4, x_2x_1, x_2^2, x_2x_4, x_3x_1, x_3x_2, x_3^2, x_4x_1, x_4x_2, x_4^2).
$$
By Theorem \ref{thm.main2'} and Theorem \ref{thm.wpapa1} and \cite[Prop 2.5]{AZ}, all the categories $\QGr A$, $\QGr kQ(A)$, $\QGr kQ'$ and $\QGr B$ are equivalent.
\end{example}


\end{document}